\documentclass[12pt,letterpaper]{amsart}

\usepackage{amsmath,amsthm,verbatim,amssymb,amsfonts,mathrsfs,amscd, graphicx,enumerate,tikz-cd,stmaryrd}
\usepackage{graphics}
\usepackage{graphicx} 

\usepackage[all]{xy}
\usepackage{fullpage}
\usepackage[pagebackref,hypertexnames=false]{hyperref}

\usepackage{etoolbox}

\usepackage{cleveref}

\usepackage{mathrsfs}
\usepackage[OT2,T1]{fontenc}
\DeclareSymbolFont{cyrletters}{OT2}{wncyr}{m}{n}
\DeclareMathSymbol{\Sha}{\mathalpha}{cyrletters}{"58}
\usepackage{enumitem}

\hypersetup{hypertexnames=false,
	colorlinks=true,
	linkcolor={red!50!black},
    citecolor={blue!50!black},
    urlcolor={blue!80!black}
}

\newcommand{\GalTom}{\operatorname{Gal}}
\newcommand{\Hom}{\operatorname{Hom}}
\newcommand{\im}{\operatorname{im}}
\newcommand{\res}{\operatorname{res}}
\newcommand{\CT}{{\operatorname{CT}}}
\newcommand{\Qbar}{{\overline{\mathbb Q}}}
\newcommand{\Z}{\mathbb{Z}}

\newcommand{\Q}{\mathbb{Q}}
\newcommand{\PP}{\mathbb{P}}

\newcommand{\LL}{\mathcal{L}}
\newcommand{\R}{\mathbb{R}}

\newcommand{\FF}{\mathbb{F}}
\newcommand{\C}{\mathcal{C}}

\newcommand{\A}{\mathbb{A}}

\newcommand{\Sel}{\mathrm{Sel}}
\newcommand{\Pic}{\mathrm{Pic}}
\newcommand{\Gal}{\mathrm{Gal}}
\newcommand{\Sym}{\mathrm{Sym}}

\newcommand{\OO}{\mathcal{O}}
\newcommand{\Ahat}{{\widehat{A}}}
\newcommand{\Bhat}{{\widehat{B}}}
\newcommand{\psihat}{{\widehat\psi}}
\newcommand{\phihat}{{\widehat\phi}}
\def \qmp {{\Q}^*/\Q^{*p}}

\theoremstyle{plain}

\newtheorem{corollary}{Corollary}
\newtheorem{theorem}[corollary]{Theorem}
\newtheorem{lemma}[corollary]{Lemma}
\newtheorem{proposition}[corollary]{Proposition}

\theoremstyle{definition}

\newtheorem{remark}[corollary]{Remark}

\newtheorem{example}[corollary]{Example}

\numberwithin{corollary}{section}
\numberwithin{equation}{section}

\usepackage{graphicx}
\makeatletter
\@namedef{subjclassname@2020}{\textup{2020} Mathematics Subject Classification}
\makeatother

\begin{document}

\title{Arbitrarily large $p$-torsion in Tate-Shafarevich groups}

\author{by E.\ Victor  Flynn}
\address{Mathematical Institute, University of Oxford,
Andrew Wiles Building, Radcliffe Observatory Quarter,
Woodstock Road, Oxford OX2 6GG, United Kingdom}
\email{flynn@maths.ox.ac.uk}
\author{Ari Shnidman, with an appendix by Tom Fisher}
\address{Einstein Institute of Mathematics,
Hebrew University of Jerusalem,
Givat Ram. Jerusalem, 9190401, Israel}
\email{ari.shnidman@gmail.com}
\address{University of Cambridge,
          DPMMS, Centre for Mathematical Sciences,
          Wilberforce Road, Cambridge CB3 0WB, UK}
\email{T.A.Fisher@dpmms.cam.ac.uk}

\date{25 September, 2023}

\begin{abstract}
 We show that, for any prime~$p$, there exist
 absolutely simple abelian varieties over~$\Q$ with arbitrarily large $p$-torsion in their Tate-Shafarevich 
groups.  To prove this, we construct explicit $\mu_p$-covers of Jacobians 
of curves of the form $y^p = x(x-1)(x-a)$ 
which violate the Hasse principle. In the appendix, Tom Fisher explains how to interpret our proof in terms of a Cassels-Tate pairing. 
\end{abstract}

\maketitle

\normalsize
\baselineskip=14pt
\section{Introduction}
\label{sec:intro}

An algebraic variety $Y$ over $\Q$ {\it violates the Hasse principle} 
if $Y(\Q) = \emptyset$ despite the fact that $Y(\Q_p) \neq \emptyset$ for all completions $\Q_p$ of $\Q$, including the archimedean completion $\Q_{\infty} = \R$.
The Hasse-Minkowski theorem shows that quadrics in $\PP^n$ never violate the Hasse principle, but violations do exist in higher degree. Some early examples include the hyperelliptic curve 
$2y^2 = x^4 - 17$ studied by Lind and 
Reichardt \cite{lind:quartic, reichardt:quartic} and Selmer's plane 
cubic $3x^3 + 4y^3 + 5z^3 = 0$ \cite{selmer:cubic}. Each of these 
is a genus one curve $C$, and is therefore 
 a torsor for its Jacobian, the elliptic curve $E = \mathrm{Pic}^0(C)$. 
The fact that $C$ violates the Hasse principle means that it represents a non-trivial 
element $[C]$ in the Tate-Shafarevich group $\Sha(E)$ parameterizing 
locally trivial $E$-torsors.  The order of $[C]$ in $\Sha(E)$ is, in these cases, equal to the minimum positive 
degree of an effective $0$-cycle, hence~$2$ 
in the first example and~$3$ in the second.

There are by now many 
other examples of non-trivial elements of Tate-Shafarevich groups of elliptic curves. However, it is 
an open question whether 
for every prime $p$ there exists an elliptic curve $E/\Q$ with a class 
of order $p$ in $\Sha(E)$. Geometrically, such $E$-torsors are realized 
as genus one curves $C \subset \mathbb{P}^{p-1}_{\Q}$ contained in 
no hyperplane, which violate the Hasse principle.  
The lack of a systematic construction of order $p$ elements is somewhat surprising, since heuristics 
of Delaunay predict that for a given prime $p$, the probability that 
a random elliptic curve $E$ satisfies $\Sha(E)[p] \neq 0$ should be 
positive \cite{delaunay:heuristics}.  

More generally, for any abelian variety $A/\Q$, the group $\Sha(A)$ parameterizes $A$-torsors which violate the Hasse principle.  Like the 1-dimensional case of elliptic curves, there are few examples with $\Sha(A)[p] \neq 0$ for large primes $p$, beyond examples where $A = \mathrm{Res}^F_\Q B$ is the Weil restriction of an abelian variety $B$ over a number field $F$ with $\Sha(B)[p] \neq 0$ (see e.g. \cite{clarksherif, kloo, kloosch}). However, the second author and Weiss \cite{shnidmanweiss} recently showed that for 
every prime $p$, there exist {\it absolutely simple} abelian varieties 
$A$ over $\Q$ with $\Sha(A)[p] \neq 0$.  They prove such $A$ exist among the quadratic twists of 
quotients of modular Jacobians $J_0(N)$ with prime level 
$N \equiv 1\pmod {p}$, but the proof does not yield explicit examples.  

\subsection{Results}
Our first main result is an explicit construction of $A$-torsors $X$ which violate the Hasse principle.  In our examples, both $A$ and $X$ have very simple equations.  
To state the theorem, recall the $p$-th power character $\left(\frac{q}{\ell}\right)_p$, which satisfies $\left(\frac{q}{\ell}\right)_p = 1$ if and only if $q$ is a $p$-th power in $\Q_\ell^\times$.

\begin{theorem}\label{thm:explicittorsor}
Let $p > 5$ be a prime and let $u,v$ be integers not divisible by $3$. 
Let~$U$ be the set of primes dividing $3puv(u-3v)$. 
Let $t \geq 2$, and let $k = p_1 p_2 \cdots p_t$, where
each~$p_i$ is a prime not in $U$ satisfying: 
\begin{enumerate}
\item $\left(\frac{p_i}{p_j}\right)_p = 1$, 
for all $i \neq j$ in $\{1,\ldots, t\}$,
\item $\left(\frac{p_i}{q}\right)_p = 1$, 
for all $i \in \{1, \ldots, t\}$ and all $q \in U$,
\item $\left(\frac{q}{p_i}\right)_p = 1$, 
for all $i \in \{1, \ldots, t\}$ and all $q \in U\backslash \{3\}$,
\item $\left(\frac{3}{p_i}\right)_p \neq 1$, for all $i \in \{1, \ldots, t\}$. 
\end{enumerate}
Let $g = p-1$ and consider the variety $\tilde A \subset \A_\Q^{2g+1}$ 
defined by the equations 
\[y_i^p = x_i(x_i - 3uk)(x_i - 9vk),\mbox{ for } i = 1,\ldots ,g,
\mbox{ and } z^p = \prod_{i = 1}^g x_i(x_i - 3uk).\]
The symmetric group $S_g$ acts on $\tilde A$ and the quotient $\tilde A/S_g$ is 
birational to a unique $g$-dimensional abelian variety $A$ over $\Q$.  Let 
$I \subset \{1, \ldots, t\}$ be a proper non-empty subset, and let 
$q = \prod_{i \in I} p_i$.  Let $\tilde X \subset \A_\Q^{2g+1}$ be defined by 
the equations $($with $i = 1,\ldots, g)$
\[y_i^p = x_i(x_i - 3uk)(x_i - 9vk) \mbox{ and } qz^p 
= \prod_{i = 1}^g x_i(x_i - 3uk).\]
Then $\tilde X/S_{g}$ is birational to an $A$-torsor $X$ that 
violates the Hasse principle, and the class of $X$ in $\Sha(A)$ has order $p$.   
\end{theorem}

\begin{remark}
Both $A$ and $X$ are $\mu_p$-covers of the Jacobian $J$ of the genus $p-1$ superelliptic curve 
$C \colon y^p = x(x-3uk)(x-9vk)$. Since $J$ is birational to 
the symmetric power $C^g/S_g$, the $\mu_p$-covers can be seen 
from the equations above as well.
\end{remark}

Using the Cebotarev density theorem, we show in Proposition \ref{prop:arbmany} that there exist primes  $p_1,\ldots, p_t$ satisfying 
the hypotheses of Theorem \ref{thm:explicittorsor}.  
Here is an example with $p = 29$.  

\begin{example}
Let $\tilde X \subset \A_\Q^{28} \times \A_\Q^{28} \times \A_\Q^1$ 
be the variety defined by the $28$ equations
\[y_i^{29} = x_i(x_i - 3\cdot 386029093\cdot545622299)
(x_i +9\cdot386029093\cdot545622299)\]
for $i = 1, \ldots, 28$, as well as the additional equation 
\[386029093  z^{29} 
= \prod_{i = 1}^{28} x_i(x_i - 3\cdot 386029093\cdot545622299).\]
Then $\tilde X/S_{g}$ is birational to a torsor $X$ for a $28$-dimensional abelian 
variety $A$ over $\Q$. Moreover, $X$ violates the Hasse principle and represents an order $29$ element of $\Sha(A)$.
\end{example}

\begin{remark}
As a point of comparison, work of Radi\v cevi\'c \cite{radicevic:explicit} 
gives a method to compute equations for order $p$ torsors 
in the Tate-Shafarevich group of an elliptic 
curve~$E$ over~$\Q$. 
Even for $p = 11$, the equations for these torsors are not so 
easy for humans to write down. As $p$ grows, the computations 
quickly become intractable even for computers. 
\end{remark}

Since the hypotheses of Theorem \ref{thm:explicittorsor} are always met, this gives a second proof of \cite[Thm.\ 1]{shnidmanweiss}, and moreover gives explicit examples for any prime $p$.  Moreover, the flexibility of the index set $I$ allows us to prove our second main result, that $\Sha(A)[p]$ can be arbitrarily large.

\begin{theorem}\label{thm:mainthm} 
For every prime~$p$ and every integer $k \geq 1$, there exists an 
absolutely simple abelian variety $A$ over~$\Q$ with $\#\Sha(A)[p] \geq p^k$.
\end{theorem}

The cases $p = 2,3,5$ not covered by Theorem \ref{thm:explicittorsor} were proven by B\"olling \cite{bolling}, Cassels \cite{cassels}, and Fisher \cite{fisher}, respectively. Indeed, it was previously known that the
$p$-part of the Tate-Shafarevich group of absolutely simple
abelian varieties over~$\Q$ can be arbitrarily large only for certain 
small primes $p$.  
%
Our examples are special since they arise as $\mu_p$-covers 
of a specific type of Jacobian, so we leave open the question 
of existence of order $p$ elements in $\Sha(A)[p]$ for ``generic'' abelian varieties 
over $\Q$, i.e.\ those such that the Mumford-Tate group is $\mathrm{GSp}_{2g}$ and $A[p]$ is irreducible as a $\Gal(\bar \Q/\Q)$-module.  In both this 
paper and \cite{shnidmanweiss}, the abelian varieties 
are such that $\mathrm{rk}\, \mathrm{End}(A_{\bar{\Q}}) = \dim A$ and $A[p]$ is reducible. 

Since we can control the dimension of our examples, we also conclude:
\begin{corollary}\label{cor:g=p-1}
Suppose $g = p-1$ for some prime $p \geq 7$. Then the Tate-Shafarevich groups of absolutely simple abelian varieties $A$ over $\Q$ of dimension $g$ can be arbitrarily large. More precisely, the groups $\Sha(A)[p]$ can be arbitrarily large. 
\end{corollary}

Our construction generalizes in an obvious way to any global field. We work over $\Q$ because it is the most interesting case and to keep the notation simple. The restriction $p \neq 5$ in our results is related to some quirky numerology (see Proposition A$(iii)$ in the Appendix) that could probably be removed by tweaking the construction slightly.

\subsection{Previous work}
Previous work on elliptic curves
(\cite{bkls:Sha, bolling, cassels, fisher, kloo, kloosch, 
kramer, lemmermeyer, lemmermollin, matsuno}),
has found arbitrarily large $p$-torsion part of the Tate-Shafarevich group
for $p \leqslant 7$ and $p=13$. In higher dimension, Creutz~\cite{creutz} has
shown that for any principally polarized abelian variety~$A$ over 
a number field~$K$, the $p$-torsion in the Tate-Shafarevich group
can be arbitrarily large over a field extension~$L$ of degree which
is bounded in terms of~$p$ and the dimension of~$A$, generalising
work of Clark and Sharif~\cite{clarksherif}.
In higher dimension over~$\Q$, the first author~\cite{flynn} has recently 
shown that the 2-torsion subgroup of Tate-Shafarevich groups of absolutely
simple Jacobians of genus~$2$ curves over~$\Q$ can be arbitrarily large, 
and then in~\cite{flynnvarieties}
that the $2$-torsion of the Tate-Shafarevich groups of absolutely
simple Jacobians of curves of any genus over~$\Q$ can be
arbitrarily large. With Bruin, the authors recently 
showed in~\cite{bfs:sqrtthree} that $\Sha(A)[3]$ can be arbitrarily 
large among certain abelian surfaces $A/\Q$.  
Many of these works make use of Jacobians
with an isogeny to another Jacobian, comparing the bound
obtained using isogeny-descent against that of a complete
$p$-descent.
\par 

\subsection{Approach}
Our method makes use of Jacobians with two independent $\Q$-rational 
$p$-torsion points, so 
we also make (implicit) use of isogenies.  However, instead of bounding the Mordell-Weil 
rank, we construct  locally soluble torsors and show directly that they have 
no rational points.   Since our method does not require knowledge 
of $L$-functions nor any information related 
to the rank of $A(\Q)$, it is more widely applicable.  Our technique 
is similar in spirit to that of Cassels in \cite{cassels} who used the Cassels-Tate pairing to show 
that the $3$-part of the Tate-Shafarevich group of elliptic curves 
can be arbitrarily large.  
However, our approach is more direct. In the appendix by Tom Fisher,
an alternative interpretation of our proof is given 
in terms of an appropriate Cassels-Tate pairing.  

We construct our torsors purely geometrically, as $\mu_p$-covers. 
In fact, we avoid the use of Galois cohomology in this paper, as a way of emphasizing the geometry. 
Experts will see that the proof can be interpreted cohomologically using standard descent 
techniques \cite{schaefer:sel, prolegom}, but the geometric point of view is 
the most direct way to understand the construction and will perhaps be 
more accessible to those less familiar with Selmer 
groups (though we do assume familiarity with the basics of abelian varieties). 

\subsection{Outline of proof}
In Section \ref{sec:mucovers} and \ref{sec:mudescent}, we prove some preliminary 
material on $\mu_p$-covers and $\mu_p$-descent. Most of this will be well-known to experts, but we have customized the discussion to our needs and made it fairly self-contained. 
In Section \ref{sec:jacobians}, we specialize the discussion to $\mu_p$-covers of Jacobians of superelliptic curves.  In Section \ref{sec:arblarge}, we prove Theorem \ref{thm:explicittorsor}.  
We must show that the torsors have local points everywhere and yet 
have no rational points. For most primes $\ell$, it is easy to see 
that the torsors have $\Q_\ell$-points using the fact that almost 
all of the primes in the set $\{p_1,\cdots, p_t\} \cup U$ are $p$-th 
powers modulo each other.  The subtle case is where $\ell = p_i$, and 
in this case we construct points explicitly using the torsion points  
$D_0 = (0,0) - \infty$ and $D_1= (3uk,0) - \infty$ on $J$.  The more 
interesting argument is the proof that the torsors have no global 
points.  For this, we first show that the two global torsion 
divisors $D_0$ and $D_1$ generate a certain quotient of $J(\Q_{p_i})$, 
for each $p_i$.  The presence of the powers of $3$ in the model of 
the curve, and the fact that $3$ is not a $p$-th power locally, 
then ``glues together'' the localizations of the torsors in a certain way that makes it impossible for them to have a global point unless 
the parameter $q$ is divisible by either all or none of the primes $p_1, \ldots, p_t$.  
The particular choice of the prime $3$ here is not special (we could 
replace it by $5$ or $7$, etc.) but the presence of this 
``gluing prime'' plays the crucial role in the argument.  

In Section~\ref{sec:proofofThm1}, we deduce Theorem~\ref{thm:mainthm} 
from Theorem~\ref{thm:explicittorsor}. First, we use a Cebotarev 
argument to show that given $p$, the set $U$, and any $t \geq 1$, 
there exist primes $p_1, \ldots, p_t$ satisfying the conditions of 
Theorem~\ref{thm:explicittorsor}. Second, the flexibility in the 
choice of $q$ allows us to generate a subgroup of $\FF_p$-rank at 
least $t - 1$ in $\Sha(A)[p]$.  Finally, we use a theorem of Masser 
to show that for $100\%$ of integers $u,v$ not divisible by~$3$, 
the corresponding abelian variety is geometrically simple. In the appendix, Tom Fisher recasts our proof in terms of a Cassels-Tate pairing. 

\subsection{Acknowledgements} The authors thank Michael Stoll for his comments and for 
organizing Rational Points 2022, where they began working together 
on this problem.  The second author was supported by the Israel 
Science Foundation (grant No. 2301/20).  The authors also thank Ariyan Javanpeykar, Jef Laga, and Ariel Weiss for comments on an earlier draft. Finally, we thank Tom Fisher for his suggestions and for  letting us include his appendix. Competing interests: The authors declare none.

\section{$\mu_p$-covers}\label{sec:mucovers}
\subsection{Classifying $\mu_p$-covers}
Let $X$ be a proper variety over a field $F$.  Let $\mu_p$ be the $F$-group scheme of $p$-th roots of unity.  
A $\mu_p$-cover of $Y$ (or more formally,  a $\mu_p$-torsor over $Y$ in the fppf topology) is a $Y$-scheme $X$ together with a $\mu_p$-action that is simply 
transitive on fibers over $Y$. The $\mu_p$-covers of $Y$ form a category $\mathcal{M}_p(Y)$ 
whose morphisms are $\mu_p$-equivariant isomorphisms.  The following 
proposition gives a concrete way to think about $\mu_p$-covers. 
   
\begin{proposition}\label{prop:equiv}
There is an equivalence of categories between $\mathcal{M}_p(Y)$ and 
the category of pairs $(\LL, \eta)$ where $\LL$ is an invertible 
sheaf on $Y$ and $\eta \colon \LL^{\otimes p} \simeq \OO_Y$ is an 
isomorphism. Here, the morphisms $(\LL, \eta) \to (\LL', \eta')$ are isomorphisms 
$g \colon \LL \to \LL'$ such that $\eta' \circ g^{\otimes p} = \eta$.  
\end{proposition}
\begin{proof}
This is well-known (see \cite{arsievistoli:stacks} or \cite[pg. 71]{mumford:abelianvarieties}), so we just describe the functors in 
both directions. If $\pi \colon X \to Y$ is a $\mu_p$-cover, then 
there is a $\Z/p\Z$-grading on the $\OO_Y$-module 
\[\pi_*\OO_X = \OO_Y \oplus \bigoplus_{i = 1}^{p-1} \LL_i\]
where each $\LL_i$ is the invertible subsheaf of $\pi_*\OO_X$ on 
which $\mu_p$ acts by $\zeta \cdot s = \zeta^i s$.  The algebra 
structure of $\pi_*\OO_X$ gives isomorphisms 
$\LL_i \otimes \LL_j \simeq \LL_{i + j}$, where indices are to be 
taken modulo $p$ and where $\LL_0 = \OO_Y$.  Thus, we obtain an 
isomorphism $\LL_1^{\otimes p} \simeq \OO_Y$.  Conversely, starting 
with a pair $(\LL,\eta)$, we can define a sheaf of $\OO_Y$-algebras 
$\OO_Y \oplus \bigoplus_{i = 1}^{p-1} \LL^i$ using the given isomorphism 
$\eta$ to define the multiplication 
$\LL^i \otimes \LL^j \simeq \LL^{i+j} \simeq \LL^{i+j - p}$ on the 
factors with $i + j \geq p$.  The relative spectrum of this sheaf 
over $Y$ is then naturally endowed with a $\mu_p$-action making it 
a $\mu_p$-cover.       
\end{proof}

\begin{remark}
If $Y = \mathrm{Spec} \, F$, this recovers Kummer theory.     
\end{remark}

\subsection{$\mu_p$-covers of abelian varieties}
Let us now specialize to the case where $Y$ is an abelian variety over a field $F$ of characteristic not $p$.  
We will think of a $\mu_p$-cover  $\pi \colon X \to Y$ in terms of the corresponding 
 pair $(\LL, \eta)$.  The isomorphism class of  $\LL$ is a well-defined element of $\Pic(Y) = \Pic_Y(F)$, called the 
{\it Steinitz class} of $\pi$.  The existence of $\eta$ means that $\LL$ is $p$-torsion, so that $\LL \in \widehat{Y}[p](F)$, where  
$\widehat{Y}  = \Pic^0_Y \subset \Pic_Y$ is the dual abelian variety parameterizing algebraically trivial line bundles on $Y$.  

From one $\mu_p$-cover $(\LL, \eta)$, we may construct many more, 
simply by scaling $\eta \colon \LL^{\otimes p} \to \OO_Y$ by 
any $r \in F^*$.  Two $\mu_p$-covers $(\LL, r\eta)$ and $(\LL, s\eta)$ 
are isomorphic  if and only if $r/s\in F^{*p}$. More generally, given 
two $\mu_p$-covers $(\LL,\eta)$ and $(\LL', \eta')$, the tensor 
product $(\LL \otimes \LL', \eta \otimes \eta')$ is another. 
Let $H^1(Y, \mu_p)$ denote the set of isomorphism classes of 
$\mu_p$-covers of $Y$. 
\begin{proposition}\label{torsorses}
The set $H^1(Y, \mu_p)$ is naturally an abelian group, and sits in 
a short exact sequence
\[ 0 \to F^*/F^{*p}  \to H^1(Y, \mu_p) \to \widehat{Y}[p](F) \to 0.\]
\end{proposition}
\begin{proof}
This follows from Proposition \ref{prop:equiv} and the discussion above.
\end{proof}

\begin{remark}
We use the notation $H^1(Y, \mu_p)$ since the \'etale cohomology group $H^1_\mathrm{et}(Y, \mu_p)$ is also in bijection with isomorphism classes of $\mu_p$-covers.   From this point of view, one obtains Proposition \ref{torsorses}  by applying the long exact sequence in cohomology to the short sequence of sheaves $0 \to \mu_p \to \mathbb{G}_m \to \mathbb{G}_m \to 0$.   
\end{remark}

\begin{lemma}
The $\mu_p$-cover $\pi \colon X \to Y$ corresponding to $(\LL, \eta)$ is geometrically connected if and only if $\LL\not\simeq \mathcal{O}_X$.  
\end{lemma}
\begin{proof}
If $\LL \simeq \mathcal{O}_X$, then $\eta$ is scalar multiplication by some $r \in F^\times$. In this case, $X$ is isomorphic to  $Y \times_F F(\sqrt[p]{r})$ as an $F$-scheme, which is not geometrically connected.  Conversely, if $X$ is not geometrically connected, then the $\mu_p$-cover $X_{\bar F} \to Y_{\bar F}$ induces an isomorphism on connected components, forcing $X_{\bar F}$ to be isomorphic to the trivial $\mu_p$-torsor $Y_{\bar F} \times_{\bar F} \mu_p$.  It follows that $\pi$ is in $\ker(H^1(Y, \mu_p) \to H^1(Y_{\bar F}, \mu_p)) \simeq F^*/F^{*p}$, hence has trivial Steinitz class.    
\end{proof}

Suppose now that $\pi \colon X \to Y$ is a geometrically connected $\mu_p$-cover corresponding to $(\LL, \eta)$, so that $\LL \not\simeq \OO_Y$. 
Since every connected finite \'etale 
cover of the abelian variety $Y_{\bar F}$ is itself an abelian variety \cite[\S 18]{mumford:abelianvarieties}, $X$ becomes an abelian variety over the algebraic 
closure $\overline{F}$. It follows that $X$ is a torsor for a certain abelian variety, which we will now identify.
 
Let $ \widehat{\psi} \colon \widehat{Y} \to \widehat{Y}/\langle \LL \rangle$ be the degree $p$ 
isogeny obtained by modding out by $\LL$. Let $\psi \colon A_\LL \to Y$ 
be the dual isogeny, which is also of degree $p$.  Then $\psi$ can itself 
be given the structure of $\mu_p$-cover. Indeed, we have  \[\ker(\psi) \simeq \widehat{\ker(\widehat{\psi})} \simeq \widehat{\Z/p\Z} 
= \mathrm{Hom}(\Z/p\Z, \mathbb{G}_m) \simeq \mu_p.\]  
Note that there are $p-1$ 
different isomorphisms $\ker(\psi) \simeq \mu_p$, corresponding to the different 
$\Z/p\Z$-gradings we can put on $\psi_*\OO_{A_\LL}$.  Exactly one of them 
will have the property that the corresponding $\mu_p$-cover has Steinitz 
class $\LL_1 \subset \psi_*\OO_{A_\LL}$  isomorphic to $\LL$. We choose 
this $\mu_p$-cover structure for $\psi$.    

\begin{lemma}\label{lem:torsor}
Let $\pi \colon X \to Y$ be a $\mu_p$-cover with non-trivial Steinitz 
class $\LL \in \widehat{Y}[p](F)$. Then $\pi$ is a twist of the 
$\mu_p$-cover $\psi \colon A_\LL \to Y$ and $X$ is a torsor for $A_\LL$.  
\end{lemma}

\begin{proof}
If $\psi \colon A_\LL \to Y$ corresponds to $(\LL, \eta)$, then 
$\pi \colon X \to Y$ corresponds to $(\LL, s \eta)$ for some scalar 
$s \in F^*$.  Over $\bar{F}$ there is an isomorphism 
$\rho \colon (A_\LL)_{\bar{F}} \to X_{\bar{F}}$ of $\mu_p$-covers, which satisfies
\[\rho^g(P) = \sqrt[p]{s}^g/\sqrt[p]{s} + \rho(P)\] 
for all $g \in \Gal(\bar F/F)$ 
and $P \in A_\LL(\bar{F})$; here $\sqrt[p]{s}^g/\sqrt[p]{s} \in \mu_p$ and $+$ is the torsor action.     The torsor structure $A_\LL \times X \to X$ is given 
by $(P, Q) \mapsto \rho(P +  \rho^{-1}(Q))$. Using the formula 
for $\rho^g$, we see that this torsor is indeed defined over $F$.  
\end{proof}

We have seen that for each non-zero $\LL \in \widehat{Y}[p](F)$, there 
is in fact a distinguished $\mu_p$-cover with Steinitz class $\LL$, 
namely the cover $A_\LL \to Y$.  This means there must be a distinguished 
isomorphism $\eta \colon \LL^p \simeq \OO_Y$.  We will describe 
this isomorphism $\eta$ in Lemma \ref{lem:distinguishedfunction}, in the context of rational points.  
For simplicity we will specialize to the case where $Y$ is a Jacobian, 
and in particular principally polarized (so that $\widehat{Y} \simeq Y$).  
However, most of what we prove can be generalized to arbitrary abelian 
varieties in a straightforward way.   


\subsection{$\mu_p$-covers of Jacobians}\label{subsec:mupjacobians}

Let $C$ be a smooth projective geometrically integral curve over $F$, 
and let $J = \Pic^0(C)$ be its Jacobian. Let $g$ be the genus of $C$, 
and hence also the dimension of the abelian variety $J$.
Let  $D \in J[p](F)$ be a divisor class of order $p$.  Let
$J \to J/\langle D \rangle$ be the quotient and let 
$\psi \colon A_D\to \widehat{J}$ be the corresponding dual isogeny, 
where $A_D$ is the dual of $J/\langle D \rangle$.   Then $\psi$ is 
a $\mu_p$-cover of $\widehat{J}$ corresponding to a pair $(\LL, \eta)$, 
as in the previous section.  

\begin{remark}As before, we may choose the $\mu_p$-cover structure on $\psi$ 
so that $\LL \in \Pic^0(\widehat{J})(F)$ is mapped 
to $D$ under the isomorphism $\hat{\hat J} \simeq J$.
\end{remark}
From now on, we identify $J$ and $\widehat{J}$ via the principal 
polarization $\lambda \colon J \to \widehat J$ coming from the theta 
divisor of the curve $C$. To make this  explicit, we assume 
that $C$ contains a rational point $\infty \in C(F)$. The theta 
divisor $\Theta \subset J$ is the subvariety of degree 0 divisor 
classes of the form $E - (g-1)\infty$, where $E$ is an effective 
divisor of degree $g -1$.  The isomorphism $J \to \widehat{J}$ 
sends $P$ to $t_P^*\OO_J(\Theta) \otimes \OO_J(\Theta)^{-1}$, 
where $t_P \colon J \to J$ is translation by $P$.  We can also 
describe $\lambda(P)$ as the line bundle on $J$ associated to the 
divisor $[\Theta - P] - [\Theta]$.

After making the identification $J \simeq \widehat{J}$, we may 
view $\psi$ as a $\mu_p$-cover of $J$ and $\eta$ as an 
isomorphism $\LL^{\otimes p} \to \OO_J$. 
By Proposition~\ref{torsorses}, we have the exact sequence
\[ 0 \to F^*/F^{*p}  \to H^1(J, \mu_p) \to J[p](F) \to 0.\]

\section{$\mu_p$-descent}\label{sec:mudescent}

We continue with our assumptions on $J = \Pic^0(C)$. We have seen that to each $D \in J[p](F)$ of order $p$, there is a corresponding $\mu_p$-cover $\psi \colon A_D \to J$ giving rise to the data $(\LL, \eta)$.   These particular $\mu_p$-covers are by construction abelian 
varieties, but general $\mu_p$-covers corresponding to pairs $(\LL, r\eta)$, for $r \in F^\times$, may only 
be torsors for abelian varieties. We characterize those which 
are abelian varieties, or equivalently, those which have rational 
points. 

\subsection{Descent over general fields}
Fix $D \in J[p](F)$  and $(\LL, \eta)$, as above. Given $P \in J(F)$, we may 
consider the $\mu_p$-cover $\psi_P  = t_P \circ \psi \colon A_D \to J$, 
where $t_P \colon J \to J$ is translation by $P$.  The $\mu_p$-cover 
$\psi_P$ is endowed with the same $\mu_p$-action as $\psi$, but different 
structure map to $J$. Since the Steinitz class is in $\Pic^0(J)$, 
it is invariant under translation, and hence $\psi_P$ and $\psi$ 
have isomorphic Steinitz classes. If $\psi_P = (\LL', \eta')$, then we 
can choose an isomorphism $\LL' \simeq \LL$, and under this isomorphism 
we have $\eta' = r_P \eta$ for some $r_P \in F^*$.  Any other choice 
of isomorphism $\LL' \simeq \LL$ differs by a scalar, so the 
element $r_P$ is well-defined up to $F^{*p}$.   

\begin{lemma}\label{lem:injective}
The map $P \mapsto r_P$ induces an injective map 
$\partial^D \colon J(F)/\psi(A_D(F)) \to F^*/F^{*p}$.   

\end{lemma}
\begin{proof} Note that $r_P\in F^{*p}$ if and only if $\psi_P$ is 
isomorphic as a $\mu_p$-cover to $\psi$. But any isomorphism of 
$\mu_p$-covers induces an isomorphism of $A_D$-torsors, and hence 
must be given by translation by $Q$ for some $Q \in A_D(F)$.  
Translation by $Q$ gives an isomorphism between these 
two $\mu_p$-covers if and only if $P  = \psi(Q)$. 
\end{proof}

For completeness, we state the following result, connecting the 
map $\partial^D$ to a boundary map in Galois cohomology:
\begin{lemma}\label{rem:groupcohom}
The map $\partial^D$ is the boundary map $J(F) \to H^1(F, \mu_p) \simeq F^*/F^{*p}$ in the long exact sequence in 
group cohomology for the short exact sequence of $\Gal(\bar F/F)$-modules 
\[0 \to \mu_p \to A_D(\bar F) \stackrel{\psi}{\longrightarrow} J(\bar F) \to 0.\]
\end{lemma}
\begin{proof}
The boundary map $J(F) \to H^1(F, \mu_p)$ sends $P \in J(F)$ to the cocycle $c \colon \Gal(\bar F/F) \to \mu_p \simeq A_D[\psi]$ given by $g \mapsto Q^g - Q$, where $Q \in A_D$ is such that $\psi(Q) = P$. We must show that this cocycle agrees with the cocycle $g \mapsto  \sqrt[p]{r}^g/\sqrt[p]{r}$, where $r = r_P$.  From the proof of Lemma \ref{lem:torsor}, we see that the $A_D$-torsors $(\LL, r\eta)$ and $(\LL, \eta)$ are isomorphic (over $\bar F$) via translation by $Q$. By the explicit formula given there, this exactly means that $Q^g - Q$ is equal to the element $\sqrt[p]{r}^g/\sqrt[p]{r} \in \mu_p$.         
\end{proof}

\begin{lemma}\label{lem:rationalpoint}
The image of $\partial^D$ is the set of $r \in F^*/F^{*p}$ such that the 
$\mu_p$-cover $(\LL, r\eta)$ has a rational point.
\end{lemma}
\begin{proof}
Every torsor in the image clearly has a rational point since it is 
isomorphic to $A_D$ as a variety. Conversely, if a $\mu_p$-cover 
$X \to J$ of the form $(\LL, r\eta)$ has a rational point then the 
underlying $A_D$-torsor is isomorphic to the trivial $A_D$-torsor up 
to translation by a point $P$. Hence $\partial^D(-P) = r$.
\end{proof}

\begin{remark}
It follows that for a $\mu_p$-cover $\pi \colon X \to J$ 
with Steinitz class $\LL$, $X$ is isomorphic to $A_D$ (as varieties) 
if and only if $\pi$ corresponds to $(\LL, r\eta)$, with $r$ in the 
image of $\partial^D$.  
\end{remark}

The following lemma is immediate from the definitions, and can be used 
to give an explicit formula for the homomorphism $\partial^D$. 

\begin{lemma}\label{lem:function1}
Let $F(J)$ be the function field of $J$ and view 
$\eta^{-1} \colon \OO_J \to \LL^p$ as a global section
of $\LL^p$.  Fix an embedding of  $\LL$ as a subsheaf of $F(J)$, so that $\eta^{-1}$ is 
a non-zero element $f$ of $F(J)$.  Let $Q$ be such that $Q$ and $Q + P$ 
are in a domain of definition for $f$. Then 
$\partial^D(P) = r_P = f(P+Q)/f(Q)$, up to $p$-th powers.     
\end{lemma}

Thinking of $\eta^{-1}$ as a function on $J$ allows us to distinguish 
the unique $\mu_p$-cover $(\LL, \eta)$ corresponding to 
$\psi \colon A_D \to J$ among all $\mu_p$-covers with Steinitz 
class $\LL$, as promised.    

\begin{lemma}\label{lem:distinguishedfunction}
The $\mu_p$-cover corresponding to $(\LL, \eta)$, which is isomorphic 
to the $\mu_p$-cover $A_\LL = A_D \to J$, is characterized  
among all $\mu_p$-covers with Steinitz class $\LL$ by the fact that the value $f(0_J)$ of the function $f = \eta^{-1} \in F(J)$ 
at $0_J$ is a $p$-th power in $F^*$. (Here we assume that $\LL$ is chosen within its isomorphism class so that $f(0_J) \in F^\times$.)      
\end{lemma}

\begin{proof}
The  $\mu_p$-cover $A_D \to J$ is distinguished 
among $\mu_p$-covers with Steinitz class $\LL$ by the fact that the 
fiber above $0$ has a rational point. Indeed, if $\pi \colon X \to J$ 
is a $\mu_p$-cover of type $(\LL, r\eta)$ with a rational point 
$Q \in X(F)$ above $0 \in J(F)$, then $\pi = \psi_P$ for some 
$P \in J(F)$, and $\pi^{-1}(0) = \psi^{-1}(-P)$. It follows that 
$P \in \psi(A_D(F))$ and hence $r$ is a $p$-th power, or in other 
words $\pi$ is isomorphic to $\psi$ as $\mu_p$-covers. 

On the other hand, the pullback of the $\mu_p$-cover $(\LL, \eta)$ on $J$ to $\mathrm{Spec}\,  F$, via the inclusion 
$\{0_J\} \hookrightarrow J$, is  
$\mathrm{Spec}\, k[z]/(z^p - h)$ where $h = f(0_J)$.  This has an 
$F$-rational point if and only if $h$ is a $p$-th power.  
\end{proof}

Let $\Sym^g C = C^g/S_g$ be the $g$-th symmetric power of $C$. Points of $\Sym^g C$ correspond to effective degree $g$ divisors $E$ on $C$.  
Recall that the map $\Sym^g C \to J$ sending $E \mapsto E - g\infty$ 
is birational \cite[Thm.\ 5.1]{milne:jacobianvarieties}, hence induces an isomorphism of function fields 
$F(\Sym^g C) \simeq F(J)$.  

\begin{lemma}\label{lem:function2}
Suppose $pD = \mathrm{div}(\tilde f)$ for some $\tilde f \in F(C)$.  
Then $\LL\simeq\OO_J(\tilde D)$ for a divisor $\tilde D$ on $J$ such 
that $p\tilde D =  \mathrm{div}(f)$, 
where $f \in F(J) \simeq F(\Sym^g C)$ is the rational function  
$f(\sum_{i = 1}^g (x_i,y_i) - g\infty) = \prod_{i = 1}^g \tilde f(x_i,y_i)$.  
\end{lemma}

\begin{proof}
Assume, for simplicity, that $D = \infty - Q$ for some $Q \in C(F)$. 
Under the polarization $J \to \widehat{J}$, the point $D$ gets sent to 
the divisor $[\Theta - D] - [\Theta]$. Note that 
\[\Theta - D = \{E + Q - g\infty \colon E \, \mbox{effective of degree } g-1\}.\]
is the locus of poles of the function $f$.  Similarly, $\Theta$ is the zero locus. Taking into account multiplicities, the divisor of $f$ 
is $p[\Theta - D] - p[\Theta]$, as claimed. The general case 
where $D = \sum_j (\infty - Q_j)$ is similar. 
\end{proof}

Finally, we will use a generalization of the map $\partial$ and 
Lemma \ref{lem:injective}.  Let $H = \{D_1, \ldots, D_m\} \subset J[p](F)$ 
be a subset of $\FF_p$-linearly independent elements. For each 
$i = 1, \ldots, m$, let $\psi_i \colon A_i \to J$ be the $\mu_p$-covers 
corresponding to $D_i$.  Let $A_H = \widehat{J/\langle H\rangle}$ and 
let $\psi_H \colon A_H \to J$ be the isogeny dual to 
$J \to J/\langle H\rangle$.  Then we have a homomorphism 
\[\tilde \partial^H \colon J(F) \longrightarrow \prod_{i = 1}^m F^*/F^{*p}\]
sending $P$ to $(\partial^{D_1}(P), \ldots, \partial^{D_m}(P))$.   

\begin{lemma}
The map $\tilde \partial^H$ induces an injection 
$\partial^H \colon J(F)/\psi_H(A_H(F)) 
\hookrightarrow \bigoplus_{i = 1}^m F^*/F^{*p}$.   
\end{lemma}

\begin{proof}
We prove this in the case $m = 2$, which is the only case we 
will use. The general case follows by 
an inductive argument.  Suppose $\tilde \partial^{D_1}(P) = 0$ and 
$\tilde \partial^{D_2}(P) = 0$, so that $P = \psi_i(Q_i)$ for some 
$Q_i \in A_i(F)$ by Lemma \ref{lem:injective}. 
Let $g_i \colon A_H \to A_i$ 
be the natural maps, of degree $p$; note that $\psi_1 g_1 = \psi_2 g_2$.  The fiber diagram

\begin{equation}\label{eq:fiberdiagram}
\begin{array}{ccc}
A_H
&\longrightarrow&A_2 \\
\downarrow & &\downarrow \\
A_1&
\longrightarrow&J,\\
\end{array}
\end{equation}
shows that there is a unique point $Q$ in $A_H(\bar F)$ such that
$g_i(Q) = Q_i$ for $i = 1,2$. 
The uniqueness of $Q$ implies that it is $\Gal(\bar F/F)$-stable and 
so we have $P = \psi_H(Q)$ with $Q \in A_H(F)$.  This shows that 
$\partial^H$ is injective.  
\end{proof}

\subsection{Descent over global fields}

Suppose now that $C$ is a curve over $\Q$.  The preceding discussion 
applies for $F =\Q$, but also for $F =\Q_\ell$ for any prime 
$\ell \leq\infty$.  Having fixed $D \in J[p](\Q)$, let 
 \[\mathrm{Sel}(A_D) \subset \Q^*/\Q^{*p}\] 
be the subgroup of classes $r$ with the property that for every 
prime~$\ell$, the class of $r$ in $\Q_\ell^*/\Q_\ell^{*p}$  is in the 
image of 
$\partial^D \colon J(\Q_\ell)/\psi(A_D(\Q_\ell)) \to \Q_\ell^*/\Q_\ell^{*p}$. In other words, an element 
of $\Sel(A_D)$ is a $\mu_p$-cover $X \to J$ with Steinitz class $D$ 
and such that $X(\Q_\ell) \neq \emptyset$ for every prime $\ell$. 

Recall that if $A$ is an abelian variety over $\Q$, then $\Sha(A)$ is 
the group of $A$-torsors which are trivial over $\Q_\ell$, for all 
primes $\ell \leq \infty$. 

\begin{proposition}\label{prop:fundexactseq}
Let $\Sha(A_D)$ be the Tate-Shafarevich group of $A_D$. There is an 
exact sequence 
\[0 \to J(\Q)/\psi(A_D(\Q)) \to \Sel(A_D) \to \Sha(A_D)[\psi] \to 0\]
where $\Sha(A_D)[\psi]$ is the kernel of the map $\Sha(A_D) \to \Sha(J)$ 
induced by $\psi$. 
\end{proposition}

\begin{proof}
The map $\Sel(A_D) \to \Sha(A_D)[\psi]$ sends the $\mu_p$-cover $X \to J$ 
to the underlying $A_D$-torsor (c.f.\ Lemma \ref{lem:torsor}). The 
cocycle $c \colon \Gal(\bar F/F) \to A_D(\bar F)$ with $c(g) = \sqrt[p]{s}^g/\sqrt[p]{s}$ determines this torsor and  
has image in $\mu_p \simeq \ker(\psi)$, so the cocycle indeed becomes 
trivial in $\Sha(J)$.  The exactness of the sequence in the middle 
follows from Lemma \ref{lem:rationalpoint}.  The exactness on the right 
can be proved using direct geometric arguments, but is most easily 
seen using Lemma \ref{rem:groupcohom}.  Since we will not actually 
use the exactness on the right, we omit the proof.        
\end{proof}

The group $\Sel(A_D)$ is isomorphic to the usual Selmer group 
\[\Sel_\psi(A_D) \subset H^1(F, A_D[\psi]) 
\simeq H^1(F, \mu_p) \simeq F^*/F^{*p}.\]  
In particular, 
it is finite. This can also be seen from the following proposition.

\begin{proposition}\label{prop:unramified}
Suppose  $\ell \neq p$ is a prime of good reduction for $J$. Then the image 
of $\partial^D \colon J(\Q_\ell) \to \Q_\ell^*/\Q_\ell^{*p}$ is the 
subgroup $\Z_\ell^*/\Z_\ell^{*p}$. 
\end{proposition}

\begin{proof}
This well-known fact follows from \cite[Prop.\ 2.7(d)]{CesnaviciusFlat} if we grant 
Lemma~\ref{rem:groupcohom}, but we will give a more geometric proof in 
the spirit of this paper. The classes in $\Z_\ell^*/\Z_\ell^{*p}$ 
represent $\mu_p$-covers $X \to J$ which are trivialized by an unramified 
field extension, hence the corresponding $A_D$-torsor $X$ 
is also trivialized by an unramified field extension.  Since $A_D$ 
has good reduction at $\ell$, the torsor $X$ has a N\'eron model 
$\mathcal{X}$ over $\Z_\ell$, which is a torsor for the N\'eron 
model $\mathcal{A}$ of $A_D$ \cite[6.5. Cor.\ 4]{blr:neronmodels}. Since any torsor for a smooth proper group scheme over $\Z_\ell$ has a $\Z_\ell$-point, it follows that such classes are in the image 
of $\partial^D$ by Lemma \ref{lem:rationalpoint}.  Conversely, any 
element $r$ in the image of $\partial^D$ corresponds to a $\mu_p$-cover 
(and $A_D$-torsor) $X$ which is abstractly isomorphic to $A_D$, and 
hence has good reduction over $\Q_\ell$. By the N\'eron mapping 
property, $X$ extends to a $\mu_p$-cover and even an $\mathcal{A}$-torsor 
over $\Z_\ell$. It follows that $r \in \Z_\ell^*$, since we can 
interpret this scalar as an automorphism of a line bundle on an 
abelian scheme $\mathcal{J}$ over $\Z_\ell$ (well-defined up 
to $p$-th powers).          
\end{proof}

We will also consider more general Selmer groups. Given a subset 
$H = \{D_1, \ldots, D_m\} \subset J[p](F)$ of linearly independent 
elements, we can define an analogous Selmer group 
$\Sel(A_H) \subset \prod_{i = 1}^m F^*/F^{*p}$ which sits in an exact sequence
\[0 \to J(\Q)/\psi_H(A_H(\Q)) \to \Sel(A_H) \to \Sha(A_H)[\psi_H] \to 0\]
and which is isomorphic to the usual Selmer group $\Sel_{\psi_H}(A_H)$.  

\section{Jacobians of curves of the form $y^p = x(x-e_1)(x-e_2)$}
\label{sec:jacobians}

\subsection{A special family of curves}
Let $p > 5$ be a prime and let $e_0, e_1,e_2$ be distinct integers.  
Let $C$ be the smooth projective curve over $\Q$ with affine model
\begin{equation}\label{eq:generalC}
y^p = (x-e_0)(x-e_1)(x-e_2).
\end{equation}
There is no loss in generality in assuming $e_0 = 0$ so we will do so.  
The affine model is itself smooth, and its complement in $C$ is a 
single rational point we call $\infty$. The  genus of $C$ is $g = p-1$.  For more details on such curves see \cite{poonenschaefer}.

Let $J$ be the Jacobian of~$C$. Note
that $J(\Q)$ has $p$-torsion of rank at least ~$2$, generated
by the three divisor classes $D_i = [(e_i,0) - \infty]$, 
for $i \in \{0,1,2\}$.  The equality of divisors 
$D_0 + D_1 + D_2 = \mathrm{div}(y)$ means that $D_0 + D_1 + D_2 = 0$ 
in $J$. Let $D = D_0 + D_1$ and define the abelian varieties 
\begin{equation}\label{eq:Adefn}
\Ahat = J/\langle D_0, D_1 \rangle
\end{equation}
\begin{equation}\label{eq:Bdefn}
\Bhat = J/\langle D \rangle
\end{equation}
and the corresponding quotient isogenies $\phihat : J \rightarrow \Ahat$ 
and $\psihat \colon J \to \Bhat$.  As before, we identify $J$ 
with its dual via the canonical principal polarization, so that 
we may write $\phi : A \to J$ and 
$\psi \colon B \to J$ for the dual isogenies.  We also 
define $A_{D_i}$ to be the dual of $J/\langle D_i\rangle$ for $i = 0,1,2$ with 
isogenies $\psi_i \colon A_{D_i} \to J$. We have $B \simeq A_{D_2}$ 
since $D = -D_2$.
\par
Let $H = \langle D_0, D_1\rangle$, and define the map
\begin{equation}\label{eq:muphihatdefn}
\begin{split}
\partial^H :
&J(\Q)/\phi\bigl( A(\Q) \bigr) \longrightarrow
\qmp \times \qmp,\\
&\left[ \sum_{j=1}^g (x_j,y_j) - g \cdot \infty \right] \mapsto
\Bigl(
\prod_{j=1}^g x_j,\
\prod_{j=1}^g (x_j - e_1)
\Bigr).
\end{split}
\end{equation}
as in Lemmas \ref{lem:function1} and \ref{lem:function2} 
of Section \ref{subsec:mupjacobians}.  
In the above definition, each $x_j,y_j \in \overline\Q$, the divisor
$\sum_{j=1}^g (x_j,y_j) - g \cdot \infty$ is Galois stable,
and the left hand side is its divisor class. That such representatives 
exist follows from the fact that $C(\Q) \neq \emptyset$ 
\cite[Prop.\ 2.7]{schaefer:sel}.  The above description of $\partial^H$
applies whenever it makes sense, that is, when all~$x_j$ and $x_j - e_1$ 
are nonzero. Every class in $J(\Q)/\phi\bigl( A(\Q) \bigr)$ can 
be represented by such a divisor \cite[pg.166]{lang:abelianvarieties}.
\par
For $i \in \{1,2,3\}$, we have similar homomorphisms
\begin{equation}\label{eq:mupsihatdefnDi}
\begin{split}
\partial^{D_i} :&J(\Q)/\psi_i\bigl( A_{D_i}(\Q) \bigr) \longrightarrow
 \qmp,\\
&\left[ \sum_{j=1}^g (x_j,y_j) - g \cdot \infty \right] \mapsto
\prod_{j=1}^g \Bigl( x_j - e_i \Bigr),
\end{split}
\end{equation}
as well as the homomorphism
\begin{equation}\label{eq:mupsihatdefn}
\begin{split}
\partial^{D} :&J(\Q)/\psi\bigl( B(\Q) \bigr) \longrightarrow
 \qmp,\\
&\left[ \sum_{j=1}^g (x_j,y_j) - g \cdot \infty \right] \mapsto
\prod_{j=1}^g \Bigl( x_j(x_j - e_1) \Bigr).
\end{split}
\end{equation}
As before, the description of these maps is for representative divisors 
for which it makes sense.  However, note that by the equation for the curve, 
we have $\partial^{D_0} \cdot  \partial^{D_1} \cdot \partial^{D_2} = 1$. 
This allows us to describe the maps $\partial^{D_i}$ even on points 
where the formula above is not well-defined. For example:

\begin{lemma}\label{lem:torsionevaluation}
We have \[\partial^H(\left[(0,0)-\infty\right]) = [ e_1^{-1}e_2^{-1}, -e_1 ]\]
and
\[\partial^H( 
\left[(e_1,0)-\infty\right]) = [e_1, e_1^{-1}(e_1 - e_2)^{-1} ].\]
\end{lemma}

In the next section, we will make critical use of the following commutative diagram
\begin{equation}\label{eq:diagrammuphihat}
\begin{array}{ccc}
J(\Q)/\phi\bigl( A(\Q) \bigr)
&\overset{\partial^H}{\longrightarrow}&\qmp\times \qmp\\
\downarrow & &\downarrow \\
J(\Q)/\psi\bigl(B(\Q)\bigr)&\overset{\partial^D}
{\longrightarrow}&\qmp,\\
\end{array}
\end{equation}
whose right vertical map is $[r_1,r_2] \mapsto r_1r_2$.  

\subsection{Models}\label{subsec:models}
There are simple birational models for the $\mu_p$-covers of $J$ 
with a given Steinitz class. For concreteness, assume that the Steinitz 
class is $D = D_0 + D_1$.  The distinguished $\mu_p$-cover with this 
Steinitz class is the cover $B \to J$.  A birational model for $J$ is given by the equations
\[y_i^p = x_i(x_i - e_1)(x_i - e_2)\]
for $i = 1,\ldots, g$, modulo the action of $S_g$.
By 
Lemmas~\ref{lem:function1},~\ref{lem:distinguishedfunction}, 
and~\ref{lem:function2}, a birational model for $B$ is 
given by the same  $g$ equations above along with the additional equation
\[z^p = \prod_{i = 1}^g x_i(x_i - e_1),\]
all modulo the action of $S_g$.
Similarly, if $r \in \Q^\times$, then the $\mu_p$-cover corresponding 
to $(\LL, r\eta)$ is given by the same equations, with the last one 
twisted by $r$:
\[y_i^p = x_i(x_i - e_1)(x_i - e_2)\]
for $i = 1,\ldots, g$ and 
\[rz^p = \prod_{i = 1}^g x_i(x_i - e_1), \]
again modulo the action of $S_g$.

\section{Arbitrarily large $p$-torsion part of the Tate-Shafarevich group}
\label{sec:arblarge}
We wish to produce examples of arbitrarily large $p$-torsion subgroups in the
Tate-Shafarevich group $\Sha(B)[p]$,
by finding elements of $\Sel(B)$ which
can be shown to violate the Hasse principle by
using $\Sel(A)$. We will choose fairly generic curves of the form 
$y^p = x(x-e_1)(x-e_2)$, but then we will twist them in a carefully 
chosen way to produce our examples.  
\par

Let 
\[C = C_{u,v} : y^p = x(x-3u)(x-9v),\] where $p > 5$ is prime
and where $u,v \in \Z$ are not divisible by~$3$. Let $J$ be the 
Jacobian of $C$ and let $A$ be the isogenous abelian variety, as 
in the previous section.
\begin{lemma}\label{lem:orderJQpip}
Suppose $q$ is a prime such that $q \equiv 1\pmod p$. 
Then $J(\Q_q)/\phi\bigl(A(\Q_q)\bigr)$ 
has order~$p^2$.

\end{lemma}

\begin{proof} The congruence condition on $q$ implies $\Q_q^*$ contains 
a primitive $p$-th root of unity $\zeta$.  Over any field containing 
$\zeta$, the automorphism $(x,y) \mapsto (x, \zeta y)$ of $C$ induces 
a ring embedding $\iota \colon \Z[\zeta] \hookrightarrow \mathrm{End}(J)$. 
The degree of $\iota(\alpha)$ is equal 
to $\mathrm{Nm}(\alpha)^2 = \#(\Z[\zeta]/\alpha)^2$.  Indeed, the degree function restricted to $\Z[\zeta]$ is a power of the norm \cite[\S 19]{mumford:abelianvarieties}, and we have $\deg([n]) = n^{2g} = n^{2[\Q(\zeta) \colon \Q]}$, so it is the square of the norm in this case.   
The kernel of $\phihat \colon J \to \Ahat$ is then equal to the kernel 
of the endomorphism $1 - \iota(\zeta)$.  It follows that $\phihat$ 
agrees with $1 - \iota(\zeta)$ up to post-composition with an 
automorphism (since they have the same degree and one factors through the other), hence the abelian varieties $A$ and $\widehat{A}$ are 
isomorphic to $J$ (over any field containing $\zeta$, and in particular 
over $\Q_{q}$).  On the other hand, we have \cite[Cor.\ 3.2]{shnidman:quadtwists}
\[\dfrac{\#J( \Q_{q} )/ \phi\bigl(A(\Q_{q})\bigr)}{\#A( \Q_{q} )[\phi]} 
= c_{q}(J)/c_{q}(A),\]
where the right hand side is the ratio of Tamagawa numbers over $\Q_{q}$.  
Since $J \simeq A$ over $\Q_{q}$, this ratio is 1. We also 
have $\#A(\Q_{q})[\phi] = p^2$, which shows that 
$\#J( \Q_{q} )/ \phi\bigl(A(\Q_{q})\bigr) = p^2$.  
\end{proof}

For each integer $k$, we will consider the curve 
\[C_k = C_{u,v,k} \colon y^p = x(x-3uk)(x-9vk).\]
Another model for $C_k$ is $k^{-3} y^p = x(x-3u)(x-9v)$, 
which shows that $C_k$ is 
a $\mu_p$-twist of the original curve $C = C_{u,v}$.  Let $J_k$, $A_k$, 
and $B_k$ be the corresponding abelian varieties for the curve $C_k$. 
These are $\mu_p$-twists of $J,A,$ and $B$ respectively. 

For two primes $q$ and $\ell$, set $\left(\frac{q}{\ell}\right)_p = 1$ 
if and only if $q$ is a $p$-th power in $\Q_\ell^\times$.  Recall the exact sequence
\[0 \to J_k(\Q)/\psi(B_k(\Q)) 
\stackrel{\partial^D}{\longrightarrow} \Sel(B_k) \to \Sha(B_k)[\psi] \to 0\]
from Proposition \ref{prop:fundexactseq}.  
\begin{proposition}\label{prop:diff}
Let~$U$ be the set of primes
that divide $3puv(u-3v)$. 
Suppose $k$ is a product of distinct primes 
$k = p_1 p_2 \cdot \ldots \cdot p_t$, where $t\geq 2$ and each prime~$p_i$ 
is not in $U$, 
and satisfies
\begin{enumerate}
\item $\left(\frac{p_i}{p_j}\right)_p = 1$, 
for all $i \neq j$ in $\{1,\ldots, t\}$,
\item $\left(\frac{p_i}{q}\right)_p = 1$, 
for all $i \in \{1, \ldots, t\}$ and all $q \in U$,
\item $\left(\frac{q}{p_i}\right)_p = 1$, 
for all $i \in \{1, \ldots, t\}$ and all $q \in U\backslash \{3\}$,
\item $\left(\frac{3}{p_i}\right)_p \neq 1$, for all $i \in \{1, \ldots, t\}$. 
\end{enumerate}
Then, for all~$i$, we have $p_i \in \Sel(B_k)$ but $p_i \notin \partial^D(J_k(\Q))$.
More generally, if $q = \prod_{i \in I} p_i^{a_i}$, 
where $I$ is any non-empty proper subset 
of $ \{1,\ldots, t\}$ and $1 \leq a_i \leq p-1$, 
then $q \in \Sel(B_k)$ but $q \notin \partial^D(J_k(\Q))$.
\end{proposition}

Note that condition $(4)$ implies that $p_i \equiv 1\pmod{p}$ for all $i$. 

\begin{proof} By Lemma \ref{lem:torsionevaluation}, we have 
\[\partial^H([(0,0)-\infty]) = [ 3^{-3} u^{-1} v^{-1} k^{-2}, -3 u k ]\]
and
\[\partial^H( 
[(3uk,0)-\infty]) = [ 3uk, 3^{-2}u^{-1}(u-3v)^{-1}k^{-2} ].\]
For $i = 1,\ldots, t$, the images of these two elements in
$\Q_{p_i}^*/(\Q_{p_i}^*)^p \times \Q_{p_i}^*/(\Q_{p_i}^*)^p$ are 
$[3^{-3}p_i^{-2},3p_i]$ and $[3p_i,3^{-2}p_i^{-2}]$ respectively, 
by our assumptions on the $p_i$.  These two elements are linearly 
independent, hence generate all of 
$\partial^H\bigl( J_k(\Q_{p_i}) /\phi(A_k(\Q_{p_i})) \bigr)$,
by Lemma~\ref{lem:orderJQpip}. 
\par
Let $[r_1,r_2]$ be in 
$\partial^H\bigl( J_k(\Q)/\phi(A_k(\Q)) \bigr)$.
We shall consider elements of $\Q^*/\Q^{*p}$ as $p$th-power-free
integers.  Note that by Proposition \ref{prop:unramified}, the 
integers $r_1$ and $r_2$ can only be divisible by primes in the 
set $\{p_1, \ldots, p_t\} \cup U$. 

If there is no~$i$ such that $\mathrm{ord}_{p_i}(r_1) \equiv \mathrm{ord}_{p_i}(1/r_2) \pmod p$, then $\mathrm{ord}_{p_i}(r_1r_2) \not\equiv 0\pmod{p}$ for every~$p_i$, and so $r_1 r_2$
cannot be of the form $\prod_{i \in I} p_i^{a_i}$, for a non-empty 
proper subset of indices $I$, and any tuple of exponents~$a_i$ 
with $1 \leq a_i \leq p-1$.
\par
Suppose there exists~$i$ such that  $\mathrm{ord}_{p_i}(r_1) \equiv \mathrm{ord}_{p_i}(1/r_2) \pmod p$.  Since 
$[3^{-3}p_i^{-2},3p_i]$ and $[3p_i,3^{-2}p_i^{-2}]$
generate all of
$\partial^H\bigl( J_k(\Q_{p_i})/
\phi(A_k(\Q_{p_i})) \bigr)$, considering only the $p_i$-adic valuation, 
we have 
\[[r_1,r_2] \equiv [p_i^{-2a},p_i^a] \cdot [p_i^b,p_i^{-2b}] 
\equiv [p_i^{b-2a}, p_i^{a-2b}]\]
for integers $a$ and $b$. So 
\[r_1r_2 \equiv p_i^{b-2a + a - 2b} = p_i^{-b-a}\]
and we must have $b \equiv -a$ (mod~$p_i$). 
In other words, the image of~$[r_1,r_2]$
in $\Q_{p_i}^*/(\Q_{p_i}^*)^p \times \Q_{p_i}^*/(\Q_{p_i}^*)^p$
must be a power of the quotient of these generators.  Thus, 
in $\Q_{p_i}^*/(\Q_{p_i}^*)^p \times \Q_{p_i}^*/(\Q_{p_i}^*)^p$ we have 
\[[r_1,r_2] = [ 3^{-4} p_i^{-3}, 3^3 p_i^3 ]^m 
= [3^{-4m} p_i^{-3m},3^{3m} p_i^{3m}],\]
for some $0 \leq m < p$.  Thus in $\Q_{p_i}^*/(\Q_{p_i}^*)^p$ we have
\[r_1^3r_2^4 = 3^{-12m + 12m} p_i^{-9m+12m} =p_i^{3m}.\]
Since $3$ is not a $p$-th power modulo $p_i$, this implies that~$3$ 
divides~$r_1^3$ and~$1/r_2^4$ to the same power.
Since, for any~$j$, the elements $[ 3^{-3} p_j^{-2}, 3 p_j ]$
and $[ 3p_j, 3^{-2}p_j^{-2} ]$ generate all of
$\partial^H(( J_k(\Q_{p_j})/\phi(A_k(\Q_{p_j}))))$,
and since we have already shown that~$3$ divides~$r_1^3$ 
and~$1/r_2^4$ to the same power, we deduce that the image of~$[r_1,r_2]$
in $\Q_{p_j}^*/(\Q_{p_j}^*)^p \times \Q_{p_j}^*/(\Q_{p_j}^*)^p$
must be a power of the quotient of these generators, and so~$p_j$ 
must divide~$r_1$ and~$1/r_2$ to the same power. This applies for all~$j$,
hence no~$p_j$ will appear in $r_1 r_2$.  So again we
have that~$r_1 r_2$ cannot be of the form $\prod_{i \in I} p_i^{a_i}$ 
with $0 < \#I < t$.
\par
Since $\partial^H \bigl( J_k(\Q)/\phi(A_k(\Q)) \bigr)$ 
maps surjectively onto 
$\partial^D \bigl( J_k(\Q)/\psi(B_k(\Q)) \bigr)$, it follows 
that each element of the form $\prod_{i \in I} p_i^{a_i}$ 
with $0 < \#I < t$ and $0 < a_i < p$ is not in 
$\partial^D \bigl( J_k(\Q)/\psi(B_k(\Q)) \bigr)$.
\par
On the other hand, we show that each~$p_i$ is in 
$\partial^D ( J_k(\Q_\ell)/\psi(B_k(\Q_\ell)) )$
for every prime~$\ell$, as follows. First note that for every
prime~$\ell \in U \cup \{p_1, \ldots, p_t\}$ except~$p_i$ itself,~$p_i$ 
is a $p$th power in $\Q_{\ell}^*$, and so $p_i$ is in 
$\partial^D\bigl( J_k(\Q_\ell)/\psi(B_k(\Q_\ell)) \bigr)$
by virtue of being $\partial^D$ of the identity. 
If $\ell \notin U \cup \{p_1, \ldots, p_t\}$, we have 
$p_i \in \partial^D ( J_k(\Q_\ell)/\psi(B_k(\Q_\ell)) )$, 
by Proposition~\ref{prop:unramified}. 
For the remaining case $\ell = p_i$, note that 
$\partial^D( [(0,0)-\infty] )= [ -3^{-2} v^{-1} k^{-1} ]$
and 
$\partial^D([(3uk,0)-\infty]) = [ 3^{-1} (u-3v)^{-1} k^{-1} ]$,
and so the first divided by the square of the second gives~$p_i$,
since all other factors are $p$th powers in $\Q_{p_i}$.
Hence~$p_i$ is in the image everywhere locally.
\end{proof}

As a corollary, we deduce the first theorem from the introduction.

\begin{proof}[Proof of Theorem $\ref{thm:explicittorsor}$]
The Theorem follows from Lemma \ref{lem:rationalpoint} and 
Proposition \ref{prop:diff}, together with the discussion of models 
in Section \ref{subsec:models}.
\end{proof}

As a second corollary, we have:

\begin{corollary}\label{cor:ppart}
Let $C_{u,v,k}$ be as in Proposition \ref{prop:diff}. 
Then $\#\Sha(B_{u,v,k})[p] \geq p^{t-1}$.
\end{corollary}
\begin{proof}  Let $B_k = B_{u,v,k}$. We have the exact sequence
\begin{equation}\label{eq:shainject}
0 \longrightarrow J_k(\Q)/\psi(B_k(\Q)) \longrightarrow \mathrm{Sel}(B_k) 
\longrightarrow \Sha(B_k)[\psi] 
  \longrightarrow 0.
\end{equation}
Moreover, we have seen that any element of the form 
$\prod_{i \in I} p_i^{a_i}$, for some proper subset 
$I \subset \{1,\ldots, t\}$ is contained in $\mathrm{Sel}(B_k)$ but does 
not lie in the subgroup $J_k(\Q)/\psi(B_k(\Q)$.  Thus,  these elements 
must map non-trivially to $\Sha(B_{k}/\Q)[\psi]$.  In fact, we see 
that the intersection of $J_k(\Q)/\psi(B_k(\Q))$ with the subgroup 
of $\mathrm{Sel}(B_k)$ generated by the elements $p_1, \ldots, p_t$ 
is at most one dimensional as an $\FF_p$-vector space.  Indeed, any 
two linearly independent elements in this subgroup can be scaled so 
that they are divisible exactly once by $p_1$, and hence their ratio 
is non-zero and not divisible by $p_1$, which would be a contradiction. 
It follows that the image of the subgroup 
$\langle p_1, \ldots, p_t\rangle$ in $\Sha(B_k)[\psi]$ has dimension 
at least $t - 1$. 
Since $\deg(\psi) = p$, we have $\Sha(B_k)[\psi] \subset \Sha(B_k)[p]$, 
which finishes the proof.
\end{proof}

\section{Proof of Theorem \ref{thm:mainthm}}\label{sec:proofofThm1}
To deduce Theorem \ref{thm:mainthm} from the results of the previous section, 
we need two extra ingredients. 

\begin{proposition}\label{prop:arbmany}
For any $u,v$  as above, and for any $t \geq 0$, there are primes 
$p_1,\dots, p_t$ satisfying the conditions of 
Proposition~$\ref{prop:diff}$.
\end{proposition}
\begin{proof}  Let $K = \Q(\zeta) = \Q(\mu_p)$. 
We prove this by induction on $t$.  If $t = 0$, then there is nothing to prove.  

Now let $t > 0$ and suppose we have found primes $p_1,\cdots, p_{t-1}$ 
satisfying the conditions. Let $k = p_1 \cdots p_{t-1}$. Let $N$ be 
the product of the primes dividing $puv(u-3v)k$, and let $\zeta_{pN}$ 
be a primitive $pN$-th root of unity.   Let $L$ be the compositum 
inside $\bar \Q$ of $\Q(\zeta_{pN})$ with all of the fields 
$\Q(\sqrt[p]{q})$, with $q$ a prime dividing $N$.  Then $L$ is an 
abelian extension of $K$ and a Galois extension of $\Q$. 
Finally, let $E = \Q(\sqrt[p]{3})$ and let $F =EL$ be the compositum 
of $E$ and $L$, which is again a Galois extension of $\Q$.  

Note that the fields $E$ and $L$ are linearly disjoint over $\Q$.  
Indeed, $E/\Q$ is totally ramified at $3$, whereas $L$ is unramified at $3$. 
Thus, we have an exact sequence 
\[0 \to (\Z/p\Z) \to \Gal(F/\Q) \to \Gal(L/\Q) \to 0\]
By the Cebotarev density theorem, there exists a prime $p_t$ (in fact, 
infinitely many such primes) whose Frobenius conjugacy class in 
$\Gal(F/\Q)$ is not trivial but restricts to the trivial class 
in $\Gal(L/\Q)$.  Let us check that $p_t$ satisfies all the 
desired properties.  

By construction, $p_t$ splits completely in any subfield of $L$.  
In particular, $p_t$ splits completely in $\Q(\zeta_{pN})$ and hence 
is a $p$-th power in $\Q_q^*$ for any prime $q$ dividing $3pN$.  
For $q \nmid 3p$, this is because $p_t \equiv 1\pmod q$, and hence 
$p_t$ is a $p$-th power in $\Q_q^*$ by Hensel's lemma.  For $q = p$, 
this is because $p^2 \mid pN$ and hence $p_t \equiv 1\pmod{p^2}$ and 
hence is a $p$-th power, again by Hensel's lemma.  For $q = 3$, this 
is because every unit in $\Z_3$ is a $p$-th power.  Similarly, $p_t$ 
splits completely in $\Q(\sqrt[p]{q})$ for all $q \mid N$, so all 
such primes $q$ are $p$-th powers modulo $p_t$.  

Finally, we check that 3 is not a $p$-th power in $\Q_{p_t}^*$.  
It is enough to show that the prime $p_t$ does not have a degree 1 prime 
above it in $E$. If it did, then because $\zeta \in \Q_{p_t}^*$, 
once the polynomial $x^p - 3$ has one root in $\Q_{p_t}^*$, it necessarily 
has all of its roots in $\Q_{p_t}^*$.  Therefore $p_t$ would split 
completely in $E$.  Since $p_t$ splits completely in $L$, this would 
mean that $p_t$ splits completely in $F = EL$. But by construction, 
the Frobenius at $p_t$ is non-trivial, so $p_t$ does not split completely.   
\end{proof}

\par
It remains to show that, for each prime~$p$, there exist
examples for which $J_{u,v,k}$ is absolutely simple, 
and hence $B_{u,v,k}$ as well.
\begin{lemma}\label{lem:existsimple}
For each $p > 5$, there exist $u,v \in \Z$ as in 
Proposition~$\ref{prop:diff}$ such that $B_{u,v,k}$ is absolutely simple
for all~$k$.
\end{lemma}
\begin{proof}
The Jacobian of  the curve $y^p = x(x-1)(x-t)$ over~$\Q(t)$
is absolutely simple since there is a value
of~$t\in \mathbb{C}$, namely $t = \frac{1}{2} + \frac{\sqrt{3}}{2} i$,
which makes the curve isomorphic over~$\mathbb{C}$
to $y^p = x^3 - 1$, and the Jacobian of the latter is known to be absolutely
simple~\cite{hazama, jedrzejak}.
By a result of Masser~\cite{masser}, for an abelian variety over~$\Q(t)$, 
the geometric endomorphism ring
for 100\% of specialisations of~$t \in \Q$ (ordered by height) is the 
same as the generic geometric
endomorphism ring. Since the generic abelian variety is geometrically 
simple, this endomorphism ring is a division ring, and hence $100\%$ 
of specializations are simple as well.  But for $t = a/b \in \Q$, a 
positive proportion have~$3$ exactly dividing~$a$ and~$3$ not dividing~$b$.
So, there are many curves $y^p = x(x-1)(x-3 v/u)$ with
$u,v \in \Z$, not divisible by~$3$, with absolutely simple Jacobian.
This is a twist of the curve
$\C_{u,v,1} : y^p = x(x-3u)(x-3^2v)$,
so there are curves of this form with absolutely simple Jacobian,
giving that $B_{u,v,1}$ is also absolutely simple. Since each
$B_{u,v,k}$ is a twist of $B_{u,v,1}$, it follows that $B_{u,v,k}$
is absolutely simple for all~$k$.
\end{proof}
\begin{proof}[Proof of Theorem~$\ref{thm:mainthm}$]
This follows from Corollary~\ref{cor:ppart}, Proposition~\ref{prop:arbmany}, 
and Lemma~\ref{lem:existsimple}.
\end{proof}

\newpage
\section*{Appendix: The Cassels-Tate pairing for $p$-coverings of Jacobians}
\smallskip
\begin{center}by \textsc{Tom Fisher}\end{center}
\medskip

The purpose of this appendix is to interpret the proof of 
Proposition~\ref{prop:diff} in terms of a certain Cassels-Tate pairing. 

Let $J/\Q$ be a Jacobian, and identify $J = \widehat{J}$ in the
usual way. Let $p > 5$ be a prime. Suppose that $J(\Q)$ contains
subgroups $\Z/p\Z$ and $(\Z/p\Z)^2$ that we take to be the kernels of
isogenies $\widehat{\psi} : J \to \widehat{B}$ and
$\widehat{\phi} : J \to \widehat{A}$. We further suppose that
$\ker \widehat{\psi} \subset \ker \widehat{\phi}$, so that
$\widehat{\phi}$ factors via $\widehat{\psi}$ to give a commutative diagram
\[ \xymatrix{ J \ar[rr]^{\widehat{\phi}} \ar[dr]_{\widehat{\psi}}
    && \widehat{A} \\
 & \widehat{B} \ar[ur]_{\widehat{\nu}} }
\]

There is then a commutative diagram
\[ \xymatrix{ A(\Q) \ar[r]^\nu \ar@{=}[d] & B(\Q) \ar[r]^-{\delta_\nu}
    \ar[d]^\psi & \Q^*/(\Q^*)^p \ar[d]^\iota \\
 A(\Q) \ar[r]^\phi \ar[d]^\nu & J(\Q) \ar[r]^-{\delta_\phi} \ar@{=}[d]
    & \Q^*/(\Q^*)^p \times
    \Q^*/(\Q^*)^p \ar[d]^\pi \\ B(\Q) \ar[r]^\psi & J(\Q) \ar[r]^-{\delta_\psi}
    & \Q^*/(\Q^*)^p } \]
where $\iota(s) = (s,s^{-1})$ and $\pi(r_1,r_2) = r_1 r_2$.

We now give the Weil pairing definition of the
Cassels-Tate pairing (see \cite[Chapter 1, Proposition 6.9]{M},
\cite[Section 12.2]{PS-CTP} or~\cite{ctpps}) simplified by the fact 
that $\pi$ has an obvious section given by $r \mapsto (r,1)$. 

The Cassels-Tate pairing
\[ \langle~,~\rangle_\CT : S^{(\psi)}(B/\Q) \times S^{(\widehat{\nu})}(\widehat{
B}/\Q) \to \Q/\Z \]
is defined as follows. We start with 
\[ r \in S^{(\psi)}(B/\Q) \subset \Q^*/(\Q^*)^p \quad \text{ and }
\quad s \in S^{(\widehat{\nu})} (\widehat{B}/\Q) \subset H^1(\Q,\Z/p\Z). \]
For each prime $\ell$ we pick $P_\ell \in J(\Q_\ell)$ with
$\delta_{\psi,\ell}(P_\ell) \equiv r \mod{(\Q_\ell^*)^p}$.
Then $\delta_{\phi,\ell}(P_\ell) = (r \xi_\ell,\xi_\ell^{-1})$ for some
$\xi_\ell \in \Q_\ell^*/(\Q_\ell^*)^p$. We define
\[  \langle r, s \rangle_\CT = \sum_{\ell} (\xi_\ell,\res_\ell s)_\ell, \]
where
\begin{equation}
 \label{locpair}
 (~,~)_\ell : H^1(\Q_\ell,\mu_p) \times H^1(\Q_\ell,\Z/p\Z) \to \Q/\Z
 \end{equation}
is the local pairing given by cup product and the local invariant map. Since the sum of the local invariants of an element in $H^2(\Q, \mu_p)$ is 0,  we have $\langle r,s\rangle_\CT= 0$ for all $r \in \delta_\psi(J(\Q))$.

{\bf Proposition A.}  \label{prop:ctp}
  With the notation and assumptions of Proposition~\ref{prop:diff}
  (noting that $\delta_\phi$, $\delta_\psi$ and $S^{(\psi)}(B/\Q)$
  are there called $\partial^H$, $\partial^D$ and
  ${\operatorname{Sel}}(B_k)$) we have
  \begin{itemize}
  \item[(i)] $p_i \in S^{(\psi)}(B/\Q) \subset \Q^*/(\Q^*)^p$ for all
    $1 \leqslant i \leqslant t$.
  \item[(ii)] Let $\chi_i \in H^1(\Q,\Z/p\Z) = \Hom(\GalTom(\Qbar/\Q),\Z/p\Z)$
    be the unique continuous character that
    factors via $\GalTom(\Q(\zeta_{p_i})/\Q)$ and satisfies
    $\chi_i(\operatorname{Frob}_3) = 1$ (this is possible
    by assumption (4)). Then \[ \chi_i - \chi_j \in
    S^{(\widehat{\nu})}(\widehat{B}/\Q) \subset H^1(\Q,\Z/p\Z) \]
    for all $1 \leqslant i,j, \leqslant t$.
  \item[(iii)] For $a_1,\ldots,a_t,b_1, \ldots, b_t \in \{0,1, \ldots,p-1\}$
    with $\sum b_j \equiv 0 \pmod{p}$ we have
    \[  \left\langle \prod_{i=1}^t p_i^{a_i} , \sum_{j=1}^t b_j \chi_j
        \right\rangle_\CT
      = \frac{-5}{p} \sum_{i=1}^t a_i b_i. \]
    In particular, since $p > 5$, if $q = \prod_{i=1}^t p_i^{a_i}$, and
    the $a_i \in \{0,1, \ldots,p-1\}$ are not all equal,
    then $q \not\in \delta_\psi(J(\Q))$.
  \end{itemize}
  \begin{proof}
 (i) See the final paragraph of the proof of Proposition~\ref{prop:diff}. \\
 (ii) The restriction $\res_\ell(\chi_i)$ is unramified for all
 $\ell \not= p_i$, and trivial for all $\ell \in U \setminus \{3\}$
 by assumption (3).
 So we only need to check the local conditions at $p_1, \ldots, p_t$ and~$3$.
 By Lemmas~\ref{lem:torsionevaluation} and~\ref{lem:orderJQpip}
 we know that $\im \delta_{\phi,p_i}$
 has order $p^2$, and the natural map $\im \delta_{\phi,p_i} \to
 \im \delta_{\psi,p_i}$ is an isomorphism. Chasing around (the local
 analogue of) the above diagram shows that $\im \delta_{\nu,p_i}$
 is trivial. It follows by Tate local duality that
 $\im \delta_{\widehat{\nu},p_i}$ is all of $H^1(\Q_{p_i},\Z/p\Z)$.
 In other words, in the definition of $S^{(\widehat{\nu})} (\widehat{B}/\Q)$,
 there are no local conditions at $p_1, \ldots, p_t$. It is
 the local condition at $3$ that forces us to take differences. \\
 (iii) We take $r = p_i$ in the above description of the pairing.
 The only primes $\ell$ that can contribute to the pairing are
 those in $\{p_1, \ldots,p_t\} \cup U$. For all such primes
 $\ell \not= p_i$ we have that $r$ is trivial in
 $\Q_{\ell}^\times/(\Q_{\ell}^\times)^p$ by assumptions (1) and (2).
 Taking $P_\ell = 0$ and $\xi_\ell = 1$ we see that 
these primes make no contribution to the pairing.

 It remains to compute the contribution at $\ell = p_i$. By Lemma \ref{lem:torsionevaluation}, we have
 \[ \delta_{\phi,\ell}(P_\ell) = ((3^{-3} p_i^{-2})^a(3p_i)^b,(3p_i)^a
   (3^{-2}p_i^{-2})^b). \]
 for some $a,b \in \Z/p\Z$. We know that $(r_1,r_2) \mapsto r_1 r_2$
 maps this to $r = p_i \in \Q_{p_i}^*/(\Q_{p_i}^*)^p$, and that
 $\Q_{p_i}^*/(\Q_{p_i}^*)^p$ is generated by $3$ and $p_i$.
 Therefore $-2a-b \equiv 0 \pmod{p}$ and $-a-b \equiv 1 \pmod{p}$.
 We solve these to give $a = 1$ and $b=-2$. Therefore 
 \[ \delta_{\phi,\ell}(P_\ell) = ( 3^{-5} p_i^{-4}, 3^{5}p_i^{5} ). \]
 and $\xi_\ell = 3^{-5} p_i^{-5}$.
 Therefore
 \[ \langle p_i, \sum b_j \chi_j \rangle_\CT =
   \sum b_j ((3p_i)^{-5}, \res_{p_i} \chi_j)_{p_i} = \frac{-5 b_i}{p}. \]
  Notice that, since $p_i \equiv 1 \pmod{p}$, evaluating the 
  local pairing~\eqref{locpair} reduces to a computation of 
  Hilbert symbols. 
The formula in the statement of the proposition follows
by linearity in the first argument.
\end{proof}


\end{document}